\documentclass[12pt]{amsart}

 \usepackage{amsthm}
\usepackage{amscd}
\usepackage[dvips]{color}
\usepackage{epsfig}
\usepackage{enumerate}
\usepackage{color}

\usepackage{amsmath}
\usepackage{amssymb}
\usepackage{caption}
\usepackage[curve]{xypic}
\usepackage{cite}
\usepackage{enumitem}
\usepackage{color}
\usepackage{url}
\usepackage[usenames,dvipsnames]{xcolor}
\usepackage{graphicx}
\usepackage{verbatim}
\usepackage{tikz}
\usepackage{tikz-cd}
\usepackage[normalem]{ulem}
\usepackage{hyperref}
\hypersetup{
     colorlinks   = true,
     citecolor    = black,
     linkcolor    = blue
}

\makeatletter 
\def\@cite#1#2{{\m@th\upshape\bfseries%
[{#1\if@tempswa{\m@th\upshape\mdseries, #2}\fi}]}}
\makeatother 

\theoremstyle{plain}
\newtheorem{theorem}{Theorem}
\newtheorem{coro}{Corollary}
\newtheorem{thm}{Theorem}[section]
\newtheorem{cor}[thm]{Corollary}
\newtheorem{prop}{Proposition}
\newtheorem{ass}{Assumption}
\newtheorem{lemma}[thm]{Lemma}

\theoremstyle{definition}
\newtheorem{defn}{Definition}

\newtheorem{prob}{Problem}

\newtheorem{rem}{Remark}

\numberwithin{equation}{subsection}
\renewcommand{\bold}[1]{\medskip \noindent {\bf #1 }\nopagebreak}

\newcommand{\nc}{\newcommand}
\newcommand{\rnc}{\renewcommand}

\nc\bA{\mathbb{A}}
\nc\bB{\mathbb{B}}
\nc\bC{\mathbb{C}}
\nc\bD{\mathbb{D}}
\nc\bE{\mathbb{E}}
\nc\bF{\mathbb{F}}
\nc\bG{\mathbb{G}}
\nc\bH{\mathbb{H}}
\nc\bI{\mathbb{I}}
\nc{\bJ}{\mathbb{J}} 
\nc\bK{\mathbb{K}}
\nc\bL{\mathbb{L}}
\nc\bM{\mathbb{M}}
\nc\bN{\mathbb{N}}
\nc\bO{\mathbb{O}}
\nc\bP{\mathbb{P}}
\nc\bQ{\mathbb{Q}}
\nc\bR{\mathbb{R}}
\nc\bS{\mathbb{S}}
\nc\bT{\mathbb{T}}
\nc\bU{\mathbb{U}}
\nc\bV{\mathbb{V}}
\nc\bW{\mathbb{W}}
\nc\bY{\mathbb{Y}}
\nc\bX{\mathbb{X}}
\nc\bZ{\mathbb{Z}}
\nc\cA{\mathcal{A}}
\nc\cB{\mathcal{B}}
\nc\cC{\mathcal{C}}
\rnc\cD{\mathcal{D}}
\nc\cE{\mathcal{E}}
\nc\cF{\mathcal{F}}
\nc\cG{\mathcal{G}}
\rnc\cH{\mathcal{H}}
\nc\cI{\mathcal{I}}
\nc{\cJ}{\mathcal{J}} 
\nc\cK{\mathcal{K}}
\rnc\cL{\mathcal{L}}
\nc\cM{\mathcal{M}}
\nc\cN{\mathcal{N}}
\nc\cO{\mathcal{O}}
\nc\cP{\mathcal{P}}
\nc\cQ{\mathcal{Q}}
\rnc\cR{\mathcal{R}}
\nc\cS{\mathcal{S}}
\nc\cT{\mathcal{T}}
\nc\cU{\mathcal{U}}
\nc\cV{\mathcal{V}}
\nc\cW{\mathcal{W}}
\nc\cY{\mathcal{Y}}
\nc\cX{\mathcal{X}}
\nc\cZ{\mathcal{Z}}

\newcommand{\ra}{\longrightarrow}
\newcommand{\M}{\mathcal{M}}

\newcommand{\strata}{\mathcal{H}}

\newcommand{\R}{\mathbb R}

\nc{\dmo}{\DeclareMathOperator}
\rnc{\Re}{\operatorname{Re}}
\rnc{\Im}{\operatorname{Im}}
\rnc{\span}{\operatorname{span}}
\dmo{\rank}{rank}
\dmo{\End}{End}
\dmo{\Hom}{Hom}
\dmo{\Jac}{Jac}
\dmo{\Id}{Id}
\dmo{\Ann}{Ann}
\dmo{\Area}{Area}
\dmo{\CP}{\bC P^1}
\dmo{\Aut}{Aut}
\dmo{\Sp}{Sp}
\dmo{\SL}{SL}
\dmo{\PSL}{PSL}
\dmo{\GL}{GL}

\title[Divergence on Average]{Divergent on average directions of Teichm\"uller geodesic flow}

\author[Apisa]{Paul~Apisa}
\author[Masur]{Howard~Masur}
\begin{document}
\maketitle

\begin{abstract}
The set of directions from a quadratic differential that diverge on average under Teichm\"uller geodesic flow has Hausdorff dimension exactly equal to one-half.
\end{abstract}

%
%

\section{Introduction}

Suppose $g_t$ is a flow on a topological space $S$.  We say that $p\in S$ is {\em divergent} if $\{g_t  p \}_{t \geq 0}$ eventually leaves every compact set of $S$. We say that $p$ is {\em divergent on average} if the proportion of time that $\{g_t  p\}_{0 \leq t \leq T}$ spends in any compact subset of $S$ tends to zero as $T$ tends to $\infty$. 

Let $\M_{g,n}$ be the moduli space of closed genus $g$ Riemann surfaces with $n$ punctures. The cotangent space of $\M_{g,n}$ coincides with the moduli space $Q_{g,n}$ of holomorphic quadratic differentials on surfaces in $\M_{g,n}$ and admits an $\SL(2, \R)$ action generated by complex scalar multiplication $r_\theta(q)=e^{i\theta}q$   and Teichm\"uller geodesic flow $g_t$. 
(see the survey paper of \cite{Z}  for the definition of the flat structure defined by a quadratic or Abelian differential and the Teichm\"uller geodesic flow). The space $Q_{g,n}$ admits an $\SL(2, \R)$-invariant stratification by specifying the number of zeros and poles and their orders of vanishing. 

For any holomorphic quadratic differential $q$  it is a consequence of Chaika-Eskin~\cite[Theorem 1.1]{EC} that the set of directions $\theta$ such that $r_\theta q$  diverges on average under the Teichm\"uller flow in its stratum has measure zero. We prove the following result.

\begin{theorem}\label{T1}
For a quadratic or Abelian differential $q$ the set of directions $\theta$ such that $r_\theta q$  diverges on average (either in its stratum or in $Q_{g,n}$) has Hausdorff dimension exactly equal to $\frac{1}{2}$. 
\end{theorem}

\noindent As a corollary, we have

\begin{coro}
Let $\mathcal{H}$ be any stratum of quadratic or Abelian differentials. 
Let $f:\mathcal{H}\to \mathbb R$ be any compactly supported continuous function. Then for any $q \in \mathcal{H}$ the set of $\theta\in [0,2\pi)$ such that 
$$\lim_{T\to\infty}\frac{1}{T}\int_0^Tf(g_tr_\theta q)dt=0$$ has Hausdorff dimension at least $\frac{1}{2}$.  
\end{coro}

\begin{rem}
For any quadratic differential the set of directions that diverge on average in $Q_{g,n}$ is  contained in the set of directions that diverge on average in the stratum. In al-Saqban-Apisa-Erchenko-Khalil-Mirzadeh-Uyanik~\cite{AAEKMU}, the authors adapted the techniques of Kadyrov-Kleinbock-Lindenstrauss-Margulis~\cite{KKLM} to show that the latter set has Hausdorff dimension at most $\frac{1}{2}$ (this result improves on results of Masur~\cite{Ma91,Ma}).  Therefore, the novelty of the current  work is establishing the lower bound of Hausdorff dimension $\frac{1}{2}$ for the set of directions that diverge on average in $Q_{g,n}$. 
\end{rem}

\begin{rem}
The methods of~\cite{AAEKMU} in fact show that the Hausdorff dimension of the set of directions that diverge on average in any open invariant subset of a stratum is at most $\frac{1}{2}$. Therefore, Theorem~\ref{T1} remains true when divergence on average is considered in a stratum with finitely many affine invariant submanifolds deleted.
\end{rem}

\begin{rem}
In the classical case of $\SL(2,\mathbb{R})/\SL(2,\mathbb{Z})$ it is known that the set of directions that diverge
on average has Hausdorff dimension $\frac{1}{2}$. In fact, the behavior of  a geodesic is determined by the continued fraction expansion of its endpoint $x=[a_0,a_1, a_2, \ldots]$ (see for instance Dani~\cite{Dani85}).  The geodesic diverges on average if and only if $\left(\prod_{i=1}^n a_i \right)^{\frac{1}{n}}$ goes to $\infty$  as $n\to\infty$ (see Choudhuri~\cite[Theorem 1.2]{Choudhuri17}) and this set has Hausdorff dimension $\frac{1}{2}$ by \cite[Theorem 1.2]{ALWW}.  In Cheung~\cite{Cheung2}, it is shown that the set of real numbers for which $a_n$ tends to $\infty$ at a certain prescribed rate has Hausdorff dimension $\frac{1}{2}$. Our construction in higher genus Teichm\"uller space is modeled on this construction.  
\end{rem}
\subsection*{Connection with Previous Results}

For any holomorphic quadratic differential $q$ on a Riemann surface $X$, every direction gives a foliation of the underlying Riemann surface. By Masur~\cite[Theorem 1.1]{Ma}, if the Teichm\"uller geodesic in a given direction is recurrent, then the foliation is uniquely ergodic. In particular, the non-uniquely ergodic directions - $\mathrm{NUE}(q)$ - are divergent directions. 

By results of Strebel~\cite{StrebelQD} and Katok-Zemlyakov~\cite{KatokZ}, the collection of directions with non-minimal flow - $\mathrm{NM}(q)$ - is countable. In~\cite{MaS}, the main theorem is that outside finitely many exceptional strata of quadratic differentials (the exceptions being the ones where every flat structure induced by a holomorphic quadratic differential has a holonomy double cover that is a translation covering of a flat torus) , there is a constant $\delta > 0$ depending on the stratum so that for almost-every quadratic differential $q$ in the stratum the set of directions with non-ergodic flow - $\mathrm{NE}(q)$ - has Hausdorff dimension exactly $\delta$. The sequence of inclusions is then
\[ \mathrm{NM}(q) \subseteq \mathrm{NE}(q) \subseteq \mathrm{NUE}(q) \subseteq D(q) \subseteq DA(q) \]
where $D(q)$ and $DA(q)$ are the set of directions that diverge (resp. diverge on average). The set $D(q)$ was shown to have measure zero in Kerckhoff-Masur-Smillie~\cite[Theorem 4]{KMS}.

The set $\mathrm{NUE}(q)$ was shown to have Hausdorff dimension at most $\frac{1}{2}$ by the main theorem of Masur~\cite{Ma}. Recently, Chaika-Masur~\cite{CM18} showed that for hyperelliptic components of strata of Abelian differentials this inequality is actually an equality for almost every Abelian differential. 

\begin{prob}
Is it the case that the Hausdorff dimension of $\mathrm{NUE}(q)$ is either $0$ or $\frac{1}{2}$ for all quadratic differentials $q$. 
\end{prob}

For all known examples, the dimension is either $0$ or $\frac{1}{2}$. Despite the fact that $\mathrm{NE}(q)$ has positive Hausdorff dimension for a full measure set of quadratic differentials (outside of finitely many exceptional strata), in each stratum there is a dense set of Veech surfaces, for which $D(q) = \mathrm{NM}(q)$ and hence is countable. The fact that $D(q)$ is positive dimensional for a full measure set of $q$ and zero-dimensional for a dense set of $q$ shows that an analogue of Theorem~\ref{T1} for divergent directions does not exist in general. 

\bold{Acknowledgements.} The authors thank Jon Chaika for suggesting this problem.  They also thank Alex Eskin and Kasra Rafi for helpful conversations. This material is based upon work supported by the National Science Foundation under Grant No. DMS-1607512. HM gratefully acknowledges their support. This material is also based upon work supported by the National Science Foundation Graduate Research Fellowship Program under Grant No. DGE-1144082. PA gratefully acknowledges their support.

%
%
%

%
%

\section{Cylinders and the Thick-Thin Decomposition}

Given a Riemann surface $X$ with a holomorphic quadratic differential there are two natural metrics on $X$ - the hyperbolic metric and the flat metric induced by the quadratic differential. Assume that the quadratic  differential has unit area. Fix $\delta > 0$ small enough so that two curves of hyperbolic-length less than $\delta$ are disjoint.

The thick-thin decomposition of the flat surface is defined as follows (for simplicity we only state the definition for an Abelian differential - $(X, \omega)$ - which will be sufficient for our purpose). Let $\gamma_k$ be the simple closed curves on $X$ of hyperbolic-length less than $\delta$. There is a geodesic representative of $\gamma_k$ in the flat metric on $(X, \omega)$. Either the flat-geodesic is unique or it is contained in a flat cylinder. In the first case, cut out the unique flat-geodesic from $(X, \omega)$ and in the second excise the entire cylinder. The resulting components are called $\delta$-thick-pieces. The size of a thick-piece is defined as the smallest flat length of a simple closed curve in the thick piece that is not homotopic to a boundary curve. These definitions are due to Rafi~\cite{Rafi07} who showed that in a thick piece the hyperbolic length of any closed curve is comparable to its flat length divided by the size of the piece. 

In Eskin-Kontsevich-Zorich~\cite[Geometric Compactification Theorem (Theorem 10)]{EKZ}, the following is established (we only state a version for Abelian differentials)

\begin{thm}\label{T:EKZ}
Let $(X_n, \omega_n)$ be any sequence of unit-area translation surfaces that are not contained in a compact subset of  a stratum of Abelian differentials. By passing to a subsequence assume that $(X_n, \omega_n)$ converges to a stable differential $\omega$ on a nodal Riemann surface $X$. Let $\delta_0$ be less than half the injectivity radius of $X$ in the hyperbolic metric on the desingularized surface. Then there is a subsequence of $(X_n, \omega_n)$ so that each thick component converges to a nonzero meromorphic quadratic differential when the flat metric on the thick component is renormalized so that its size is one.
\end{thm}


\begin{defn}
Suppose that $A$ is an annulus around a curve $\gamma$ on $X$. The annulus $A$ is called regular if it is of the form $\{ p : d(p, \gamma) < r \}$ for some $r$ where $d(\cdot, \cdot)$ denotes distance in the flat metric. The annulus is primitive if additionally it contains no singularities in its interior. If $A$ is a primitive regular annulus that is not a flat cylinder, then define $\mu(A) := \log \left( \frac{|\gamma_o|}{|\gamma_i|} \right)$ where $|\cdot|$ denotes flat length and $\gamma_o$ (resp $\gamma_i$) is the longer (resp. shorter) boundary curve of $A$ and is called the outer (resp. inner) curve of $A$. This definition agrees with the one made in Minsky~\cite{Min92} up to a multiplicative constant that only depends on the stratum containing $(X, \omega)$. 
\end{defn}

\begin{lemma}\label{L:size2}
Under the hypotheses of Theorem~\ref{T:EKZ}, every flat cylinder in $(X_n, \omega_n)$ around a $\delta_0$-hyperbolically short curve has the length of its core curve tend to $0$ as $n \ra \infty$. 
\end{lemma}
\begin{proof}
By Maskit~\cite[Corollary 2]{Maskit}, since the $\delta_0$-hyperbolically short curves have lengths tending to zero in the hyperbolic metric, their extremal length also tends to zero as $n \ra \infty$. The extremal length of a curve $\gamma$ is defined to be the reciprocal of the modulus of the largest topological annulus embedded in the hyperbolic surface whose waist curve is freely homotopic to $\gamma$.

Therefore, each $\delta_0$-hyperbolically short curve $\gamma$ is contained in a topological annulus whose modulus tends to $\infty$ as $n \ra \infty$. By Minsky~\cite[Theorem 4.5 and 4.6]{Min92} (note that the inequality $\leq m_0$ should be $\geq m_0$ in the statement of the Theorem 4.6), either $\gamma$ is contained in a flat cylinder whose modulus is unbounded in $n$ or there is a primitive regular annulus $A_n \subseteq (X_n, \omega_n)$ contained in a thick piece whose core curve is homotopic to $\gamma$ and so that $\mu(A_n)$ tends to $\infty$ as $n$ increases.

Notice that if the modulus of the flat cylinder containing $\gamma$ tends to $\infty$ then the flat length of $\gamma$ tends to zero since each $(X_n, \omega_n)$ is unit-area. Therefore, suppose that for each $n$ there is a primitive regular annulus $A_n$ whose core curve is homotopic to $\gamma$ and so that $\mu(A_n)$ is unbounded in $n$.

Let $\ell_n$ be the flat length of $\gamma$ on $(X_n, \omega_n)$ and let $a_n$ be the area of the thick piece containing $A_n$. The flat distance across $A_n$ is at most $h_n:=\frac{a_n}{\ell_n}$. The flat length of the outer curve $A_n$ in the flat metric is at most $2 \ell_n + 2\pi M h_n$ where $M$ is some integer only depending on the stratum. Therefore, $\mu(A_n) \leq \log \left(2 + \frac{2\pi M}{\ell_n} \right)$.  Since $\mu(A_n)$ is unbounded in $n$, $\ell_n \ra 0$ as $n \ra \infty$. 
\end{proof}

A similar argument to the one above is given in Choi-Rafi-Series~\cite[Corollary 5.4]{CRS08}. 

\begin{defn} Fix $\delta > 0$ and $c \in (0, 1)$. A cylinder is $(\delta,c)$-thin if its circumference is at most $\delta$ and its area is at least $c$.  A half-translation surface is said to belong to the $(\delta, c)$-thick part of a stratum if it contains no $(\delta,c)$-thin cylinders. 
\end{defn}

\begin{rem}
We remark that  the $(\delta,c)$-thick set is {\em not} compact.  One can have a sequence of surfaces containing cylinders of circumferences going to $0$ and areas less than $c$ that lie in the $(\delta,c)$-thick part. These sequences enter what is usually referred to as the thin set, i.e. the set where a curve is short with no reference to area.  The basic part of our construction will be given a cylinder $\beta$ of reasonably large area, to use $\beta$  to find lots of directions that may enter the $(\delta,c)$-thin part,  but then return to the  thick set allowing us to find further cylinders of large area.  We use these sequences of   cylinders to build a Hausdorff dimension $\frac{1}{2}$ set of directions whose geodesics spend most of their time in the usually defined thin set where a curve is short.
 \end{rem}

%
%

\begin{lemma}\label{L:find-cylinder}
Fix a stratum $\strata$ of quadratic differentials, an open set $I$ on the unit circle, a positive constant $\delta$, and positive constants
\begin{enumerate}
\item $c_1 < \frac{1}{3g-3}$  
 \item $c_2 < \frac{\lambda}{g\left(2g + |s| - 2\right)}$ where $s$ is the singular set and where $\lambda = 1 - (3g-3)c_1$.
\end{enumerate}
Then there is an $L$ so that for any $(\delta, c_1)$-thick unit area surface in $\strata$ there is a cylinder of area at least $c_2$ whose core curve has length at most $L$ and whose holonomy makes an angle lying in $I$. 
\end{lemma}
\begin{proof}

Let $\strata_{(\delta, c_1)}$ be the locus of $(\delta, c_1)$-thick translation surfaces in $\strata$. By Vorobets \cite[Theorem 1.5]{Vorobets03}, for every unit area translation surface in $\strata$ there is a cylinder of area at least $\frac{1}{2g+|s|-2}$ whose core curve has holonomy that makes an angle of $I$ (call this set of cylinders $\cC$). Let $\ell: \strata \ra \R$ be the function that records the shortest length of a cylinder in $\cC$.  Since cylinders persist on open subsets of $\strata$, it follows that $\ell$ is bounded on compact subsets of $\strata$. 

Suppose to a contradiction that the claim fails. It follows that there is a sequence $(X_n, \omega_n)$ of translation surfaces in $\strata_{(\delta, c_1)}$ so that the $\ell\left( X_n, \omega_n \right) \ra \infty$. Since $\ell$ is bounded on compact subsets of $\strata$, it follows that $(X_n, \omega_n)$ leaves all compact subsets of $\strata$. By passing to a subsequence we suppose that the sequence converges to $(X, \omega)$ in the geometric compactification. 

Let $\delta_0$ be less than half the injectivity radius of the desingularized hyperbolic metric on $X$, and suppose that the $\delta_0$ thick pieces converge as in Theorem~\ref{T:EKZ}. We claim that there is a thick piece that has definite area on each $(X_n, \omega_n)$. In the thick-thin decomposition, the only positive-area subsurfaces that are not contained in a thick piece are flat cylinders around $\delta_0$-hyperbolically short curves. However, by Lemma~\ref{L:size2} these cylinders have the length of their core curves tend to zero along the sequence $(X_n, \omega_n)$. By truncating an initial segment of the sequence we may suppose that these core curves are always less than length $\delta$ in flat length. Since $(X_n, \omega_n)$ are $(\delta, c_1)$-thick surfaces we have that the thin part has area at most $(3g-3)c_1$. Therefore, the thick part has area at least $\lambda$. Moreover, the thick part has at most $g$ components. In particular, it contains some translation surface of area $\frac{\lambda}{g}$. 

By Vorobets \cite[Theorem 1.5]{Vorobets03}, this thick piece contains a cylinder in direction $I$ of area at least $\frac{\lambda}{g\left(2g + |s| - 2\right)}$. Therefore, pulling it back to $(X_n, \omega_n)$ (after again truncating a finite initial subsequence), produces a cylinder in $\cC$ along the subsequence of bounded length, which is a contradiction. 
\end{proof}

\begin{prop}\label{P:theta}
Let $\strata$, $c_1$, $c_2$, and $\delta$ be as in Lemma \ref{L:find-cylinder} and let $0 < \theta_1 < \pi$. Suppose too that $c_1 \leq c_2$. Then there are positive constants $L$ and $\theta_0$ so that the following holds.  Let $\cC$ be the collection of cylinders of area at least $c_2$ and circumference of length strictly less than $L$. For any $(\delta, c_1)$-thick surface there is a cylinder in $\cC$ that is bounded away from the horizontal by at least $\theta_1$ and bounded away from any shorter cylinder in $\cC$ by at least $\theta_0$. 
\end{prop}
\begin{proof}
Choose $\epsilon > 0$ so that $\theta_1 + \epsilon < \pi$ and let $I$ be the collection of angles that are at least $\theta_1+\epsilon$ from the horizontal. Let $L$ be the length produced by Lemma \ref{L:find-cylinder}. 

Suppose to a contradiction that $(X_n, \omega_n)$ is a sequence along which the conclusion fails. Suppose without loss of generality that the sequence has a limit $(X, \omega)$ in the geometric compactification.  Let $\delta_0$ be half the injectivity radius of $X$ in the hyperbolic metric. 

By the proof of Eskin-Kontsevich-Zorich~\cite[Geometric Compactification Theorem (Theorem 10)]{EKZ}, for sufficiently large $n$ and after passing to a subsequence there is a triangulation of the thick pieces of $(X_n, \omega_n)$ so that for constants $C_1$ and $C_2$,
\begin{enumerate}
\item The triangulations have the same combinatorial type for all $n$. 
\item There are fewer than $C_1$ edges of the triangulation. 
\item The edges of the triangulation are saddle connections.
\item The saddle connections that do not belong to the boundary of the thick piece have length bounded below by $\frac{\lambda}{2}$ and above by $C_2 \lambda$ where $\lambda$ is the size of the thick piece. 
\item The holonomy of the edges of the triangulation converge as $n$ tends to infinity. 
\end{enumerate} 
Assume again after passing to a subsequence that the sizes of each thick piece converge. 

\noindent \textbf{Step 1: For sufficiently large $n$, a cylinder in $\cC$ does not intersect a thick piece whose size tends to zero}

Each cylinder in $\cC$ has height at least $\frac{c_2}{L}$. Take $n$ sufficiently large so that for each thick piece whose size is tending to zero it is triangulated by saddle connections of length less than $\frac{c_2}{L}$. Then $\cC$ cannot intersect this thick piece since it cannot cross a saddle connection of length less than $\frac{c_2}{L}$. 

\noindent \textbf{Step 2: For sufficiently large $n$, a cylinder in $\cC$ does not intersect a positive area thin piece}

These thin pieces are exactly flat cylinders around $\delta_0$-hyperbolically short curves. By Lemma \ref{L:size2}, the circumference of these flat cylinders tend to zero. Since each $(X_n, \omega_n)$ is $(\delta, c_1)$-thick, these cylinders must have area strictly less than $c_1 \leq c_2$ for sufficiently large $n$.  Hence no cylinder in $\cC$ can coincide with one of these cylinders. Moreover, the heights of the cylinders in $\cC$ are bounded below and so no cylinder in $\cC$ can cross them either (since the circumferences tend to zero). 

\noindent \textbf{Step 3: For sufficiently large $n$, a cylinder in $\cC$ is contained in a single thick piece}

Consider a thick piece whose size does not tend to zero. By Eskin-Kontsevich-Zorich~\cite[Geometric Compactification Theorem (Theorem 10)]{EKZ}, when this piece is rescaled by its size it converges to a meromorphic quadratic differential. However, since the size is bounded away from zero this quadratic differential has finite area, no boundary, and trivial linear holonomy, i.e. it is an Abelian differential on a closed Riemann surface. Therefore, the boundary of the thick piece necessarily consisted of saddle connections whose holonomy tended to zero as $n$ tended to $\infty$. Since cylinders in $\cC$ have height that is bounded below, it must be the case that $\cC$ cannot cross the boundary of a thick piece when $n$ is sufficiently large (i.e. when all the saddle connections in the boundary are sufficiently small). 

\noindent \textbf{Step 4: There is a finite collection of curves $S$ defined only in terms of the combinatorial type of the triangulation so that any cylinder belonging to $\cC$ on $(X_n, \omega_n)$ for sufficiently large $n$, has core curve homotopic to a curve in $S$.}

Recall that the triangulations of the thick pieces of $(X_n, \omega_n)$ all have the same combinatorial type and that all edges converge in length. Let $\ell$ be the supremum of the edge lengths that appear in these triangulations for all $(X_n, \omega_n)$. Since the smallest that the height of a cylinder in $\cC$ can be is $\frac{c_2}{L}$, a cylinder in $\cC$ can only intersect an edge of the triangulation $\frac{\ell L}{c_2}$ times. Consider the finite collection $S$ of all paths through the triangulations of these thick parts that are (1) straight lines in each triangle, (2) connect midpoints of edges of the triangulation, and (3) cross any edge at most $\frac{\ell L}{c_2}$ times. For sufficiently large $n$, the core curves of the cylinders in $\cC$ are homotopic to a curve in $S$ and moreover, since the curves in $S$ are only defined in terms of the combinatorial type of the triangulation, they define piecewise geodesic curves on all $(X_n, \omega_n)$. Moreover, since the holonomy of the edges of the triangulation converge as $n$ tends to $\infty$, so does the holonomy of each element of $S$. 

\noindent \textbf{Step 5: One may choose a cylinder for each $n$ to derive a contradiction.}

Let $C_n$ be the shortest cylinder in $\cC$ whose holonomy belongs to $I$ (such a cylinder exists by Lemma \ref{L:find-cylinder}). After passing to a subsequence we may assume that $C_n$ corresponds to a fixed element $s_0 \in S$ for all $n$. Let $h_n(s_0)$ be the holonomy of $s_0$ on $(X_n, \omega_n)$. If the argument of $h_n(s_0)$ converges to an angle $\theta$ in the interior of $I$, then on each $(X_n, \omega_n)$ we notice that by choosing $C_n$ on $(X_n, \omega_n)$ we have produced a cylinder in $\cC$, whose period makes an angle of $\theta_1$ from the horizontal, and which, for sufficiently large $n$, makes an angle of $\frac{d}{2}$ from any shorter cylinder in $\cC$ where $d$ is the distance from $\theta$ to the boundary of $I$. However, $(X_n, \omega_n)$ was chosen so that along the sequence no such cylinder could be found. This is a contradiction. 

Suppose now that the argument of $h_n(s_0)$ converges to the boundary of $I$, without loss of generality suppose that it converges to $\theta_1 + \epsilon$.  Let $S'$ be the subset of $S$ of paths whose holonomy has its argument converge to $\theta_1 + \epsilon$. Define
\[ \Lambda := \left\{ \lim_{n \ra \infty} \mathrm{arg} \left( h_n(s) \right) \right\}_{s \in S} \]
Let $d$ be the distance from $\theta_1 + \epsilon$ to the nearest distinct point in $\Lambda$ (and let it be $2\pi$ if there are no other distinct points). For each $n$, let $C_n$ be the shortest cylinder in $\cC$ whose core curve is homotopic to an element of $S'$. Notice that unlike in the previous case, the argument of the holonomy of $C_n$ might lie outside of $I$.  However, for sufficiently large $n$, the argument of the holonomy of $C_n$ will be bounded away the horizontal by $\theta_1$ and from any other element of $\cC$ by at least $\frac{d}{2}$. Again this is a contradiction. 
\end{proof}

\section{The Child Selection Process}
Make the following definition and assumption.

\begin{defn}
For any real numbers $t$ and $\theta$ set
\[ g_t := \begin{pmatrix} e^t & 0 \\ 0 & e^{-t} \end{pmatrix} \qquad r_\theta := \begin{pmatrix} \cos \theta & -\sin \theta \\ \sin \theta & \cos \theta \end{pmatrix} \]
\end{defn}

\begin{ass}\label{A:basic1}
Fix a translation surface $(X, \omega)$ in a stratum $\strata$ of Abelian differentials on genus $g$ Riemann surfaces. Let $|\Sigma|$ denote the number of zeros of $\omega$. Fix $c$ so that setting $c_1 = c_2 = c$ satisfies the conditions of Lemma \ref{L:find-cylinder}. Fix $\delta, \theta_1$, and $M > 1$ so that $(X, \omega)$ is $(\delta,c)$-thick. Let $L$ and $\theta_0$ be the constants associated with $\strata, \delta, c$, and $\theta_1$ as in Proposition \ref{P:theta}.
\end{ass}

In the sequel, we will put successive conditions on the constants in this assumption. The following process - the child selection process - is the main object of study in the sequel. 

\begin{defn}[Child Selection Process] \label{D:CS}
 Consider the following process, which we will call the child selection process:
\begin{enumerate}
\item Let $\beta_0$ be a cylinder on $(X_0, \omega_0) := (X, \omega)$ of area at least $c$. Let $s$ be the saddle connection  determined  by the shortest cross curve of $\beta_0$ that makes an acute angle with the core curve.

\begin{defn}
Identifying $s$ and $\beta_0$ with their holonomy vectors on $(X_0, \omega_0)$ define $s_t := s + t \beta_0$ where $t$ is a real number contained in  $\left(\frac{2 \log|\beta_0|}{\delta}, \frac{M \log|\beta_0|}{2L} \right)$

\end{defn}

Note $s_t$ will define a saddle connection when $t$ is an integer.  In that case $s_t$ is formed by Dehn twisting $s_0$ $t$ times.  
\item Fix $t$ and rotate so that $s_t$ is vertical. Flow by $g_t$ so that $s_t$ has unit length. Call $s_t$ a protochild of $\beta_0$ and the new surface the protochild surface. Suppose that the protochild surface - call it $(X_1, \omega_1)$  - is $(\delta, c)$-thick. 
\item By Proposition \ref{P:theta} there is a cylinder $\beta_1$ on $(X_1, \omega_1)$ whose circumference is at most $L$, area is at least $c$, and whose core curve makes an angle of at least $\theta_1$ with the horizontal. Call this cylinder the child of $\beta_0$ corresponding to $s_t$. Similarly call $\beta_0$ the parent. We will use $\beta_1$ to refer to the cylinder on both $(X_0, \omega_0)$ and $(X_1, \omega_1)$. 
\end{enumerate}
\end{defn}

\begin{rem}
There are two major obstacles to iterating the child selection process. The first is that it is not clear that it is possible to find a cylinder and a protochild whose protochild surface remains $(\delta, c)$-thick. The second is that, even if one such direction exists, it is unclear that another can be found on the protochild surface. The first obstacle is addressed in Section \ref{S:started} and the second in Sections \ref{S:definite} and \ref{S:comparable}. The spacing between protochild directions is controlled in Section \ref{S:facts} and this is used to compute a lower bound on the Hausdorff dimension of divergent on average directions in Section \ref{S:T1}
\end{rem}

%
%

\section{A definite proportion of protochild surfaces are $(\delta, c)$-thick.}\label{S:definite}

Throughout this section we will continue to assume Assumption \ref{A:basic1} as well as the following:

\begin{ass} \label{A:starting1}
Suppose that $\delta < \frac{c}{576\sqrt{2}(g-1)}$ and suppose that $\beta_1$ is a parent cylinder on a $(\delta, c)$-thick surface $(X, \omega)$ such that $\log |\beta_1| > 1$.  Choose $M$ to be of the form $\frac{2^{m+2} L}{\delta}$ for a positive integer $m$.  This means that the interval of times in which a protochild is selected has the form  $\left( \frac{2 \log |\beta_1|}{\delta}, 2^{m} \frac{2 \log |\beta_1|}{\delta} \right)$. 
\end{ass}

\begin{defn}
Given two straight line segments $s$ and $s'$ in $(X, \omega)$ we will let $\theta(s, s')$ denote the angle between them. 
\end{defn} 

\begin{rem}
In the sequel, all lengths and angles will be measured on $(X, \omega)$ unless otherwise stated.
\end{rem}

\begin{lemma}\label{L:angle-between} 
Suppose that $s_t$ is a protochild of $\beta_1$ and that $\beta_2$ is a $(\delta, c)$-thin cylinder on the corresponding protochild surface. If on $(X, \omega)$ we have $\frac{|\beta_1|}{2\sqrt{2}} \leq |\beta_2|$, then $\beta_2$ cannot be parallel to $\beta_1$.
\end{lemma}
\begin{proof}
We have 
\[ \sin \theta(\beta_2, s_t) \leq \frac{\delta}{|\beta_2||s_t|} \leq \frac{ 2\sqrt{2} \delta }{|\beta_1||s_t|}< \frac{c}{|\beta_1||s_t|} \leq \sin \theta(\beta_1, s_t) \]
This implies $\beta_2$ cannot be parallel to $\beta_1$
\end{proof}


\begin{defn}\label{D:thin-set}
If $\beta$ is a $(\delta, c)$-thin cylinder on the protochild surface of  $s_t$ and $\frac{|\beta_1|}{2\sqrt{2}} \leq |\beta|$ then make the following definitions, 
\begin{enumerate}
\item By Lemma \ref{L:angle-between}, let $t_0$ be the real number such that $\beta$ points in the direction of $s_{t_0}$. 
\item Let $I(\beta, r)$ (resp. $I^h(\beta, r)$, $I^v(\beta, r)$) be the collection of $t$ for which $\beta$ has circumference (resp. with horizontal, vertical part) less than $r$ on the protochild surface corresponding to $s_t$. 
\item Define $I_1(\beta) := I(\beta, \delta)$, i.e. the collection of $t$ so that $\beta$ is $(\delta, c)$-thin on the protochild surface corresponding to $t$. Define $I_2(\beta) := I\left(\beta, \frac{c}{32} \right)$. Define $I_1^h$, $I_1^v$, $I_2^h$, and $I_2^v$ analogously.
\end{enumerate}
Note that it is possible that $t_0\notin I_1(\beta_1)$.
\end{defn}


 \begin{lemma} \label{L:basic-interval}
 Using the same notation as in Definition \ref{D:thin-set},
\[ I^h_1(\beta)= \left\{t:|t-t_0|<\frac{\delta |s_{t_0}|}{\mathrm{area}(\beta_1)|\beta|} \right\} \]
and similarly 
\[ I^h_2(\beta)=\left\{t:|t-t_0|\leq \frac{c|s_{t_0}|}{32\mathrm{area}(\beta_1)|\beta|} \right\}. \]
\end{lemma}
\begin{proof}
Let $t$ be a protochild. Rotate so $s_t$ is vertical and let $h$ be the horizontal component of $\beta$. Now  $$\frac{h}{|\beta|}=\sin\theta(s_t,\beta)=\sin\theta(s_{t_0},s_t)=\frac{\mathrm{area}(\beta_1)|t-t_0|}{|s_{t_0}||s_t|}.$$
 We now apply $g_t$ until $s_t$ has unit length.  The number $t$ belongs to $I_1^h(\beta)$ if and only if $h|s_t|<\delta$, equivalently,
 \[ \frac{\mathrm{area}(\beta_1)|t-t_0| |\beta|}{|s_{t_0}|}<\delta. \]
This proves (1).  The proof of (2) is identical.
\end{proof}

\begin{rem}\label{R:size-estimate}
Notice that when $|\beta_1| > 1$, which follows from Assumption \ref{A:starting1}, we have 
\[ t |\beta_1| \leq |s_t| \leq (t+1)|\beta_1|. \]
If $\frac{|\beta_1|}{2\sqrt{2}} \leq |\beta|$, then by Lemma \ref{L:basic-interval} the radius of $I^h(\beta, r)$ is bounded above by $\frac{4(t_0+1) r}{\sqrt{2} c}$. If $t_0 > 1$ then the bounds simplify to
\[ |s_t| \leq 2 |\beta_1| \quad \text{and} \quad \left| I^h(\beta, r) \right| \leq 2 \left( \frac{8t_0r}{\sqrt{2} c} \right) \]
\end{rem}

\begin{lemma} \label{L:interval2}
If $\frac{|\beta_1|}{2\sqrt{2}} \leq |\beta|$ and $I_1(\beta)$ intersects the interval $\left( \frac{2 \log |\beta_1|}{\delta}, 2^{m} \frac{2 \log |\beta_1|}{\delta} \right)$, then $I^h_2(\beta) \subseteq I_2^v(\beta)$.
\end{lemma} 
\begin{proof}
Let $v_t(\cdot)$ and $h_t(\cdot)$ denote respectively the vertical and horizontal parts of a holonomy vector when $(X, \omega)$ is rotated so that $s_t$ is vertical. 

If $I_1(\beta)$ intersects the interval of times in which protochildren are chosen, then so does $I_1^h(\beta)$. By Remark \ref{R:size-estimate}, it follows that 
\[ \frac{2 \log |\beta_1|}{\delta} <  t_0 + \frac{4 \delta (t_0+1)}{\sqrt{2} c} < \left( t_0 + 1 \right) \left( 1 + \frac{1}{144} \right). \]
where the second inequality follows from Assumption \ref{A:starting1}. Since $\log |\beta_1| > 1$ and $\delta < 1$ (see Assumption \ref{A:starting1}), it follows that $t_0 > 1$, so we can use the simpler bounds in Remark \ref{R:size-estimate}.  

Notice that $\lim_{t \ra \infty} \frac{v_t(\beta)}{|s_t|} = 0$. Therefore, it suffices to show that if $q$ is a point so that $\frac{v_q(\beta)}{|s_q|} = \frac{c}{32}$, then $q$ is less than the left endpoint of $I_2^h(\beta)$, which is at least $t_0 \left( 1 - \frac{1}{4 \sqrt{2}} \right)$ by Remark \ref{R:size-estimate}. Suppose to a contradiction that there is a point $q > t_0 \left( 1 - \frac{1}{4 \sqrt{2}} \right)$ such that $\frac{v_q(\beta)}{|s_q|} = \frac{c}{32}$.


Since $I_1(\beta)$ is nonempty, there is some point $p$ so that $\frac{v_p(\beta)}{|s_p|} = \delta$ and $p$ is smaller than the left endpoint of $I_1^v(\beta)$, i.e.
\[p \leq t_0 \left( 1 + \frac{8 \delta}{\sqrt{2} c} \right) < t_0 \left( 1 + \frac{1}{144} \right).\]
Notice that 
\[ \sin \theta\left( s_p, s_{t_0} \right) = \frac{|p-t_0| \mathrm{area}\left(\beta\right)}{|\beta_1||s_{t_0}|} < \frac{1}{144} \]
which implies that $p$ is close enough to $t_0$ that $v_p(\beta) \geq \frac{|\beta|}{2}$.

\noindent Since $\frac{v_q(\beta)}{|s_q|}  = \frac{c}{32}$ and $\frac{v_p(\beta)}{|s_p|} = \delta$, we have
\[ |s_q| v_p(\beta) = \frac{32\delta}{c} |s_p| v_q(\beta).  \]
This implies that
\[ \frac{q|\beta_1| |\beta|}{2} < \frac{64 \delta p |\beta_1| |\beta|}{c}. \] 
In other words, 
\[ q < \frac{128 \delta p}{c} < t_0 \left( \frac{1 + \frac{1}{144}}{2} \right) < t_0 \left( 1 - \frac{1}{4\sqrt{2}} \right) \]
which is a contradiction. Therefore, we have that every point in $I_2^h(\beta)$ is also contained in $I_2^v(\beta)$ as desired. 
\end{proof}

\begin{cor}\label{C:interval-estimate}
If $\frac{|\beta_1|}{2\sqrt{2}} \leq |\beta|$ and $I_1(\beta)$ intersects the interval $\left( \frac{2 \log |\beta_1|}{\delta}, 2^{m} \frac{2 \log |\beta_1|}{\delta} \right)$, then 
\[ I^h\left( \beta, \frac{c}{32\sqrt{2}} \right) \subseteq I_2(\beta) \subseteq I_2^h(\beta). \]
It follows that
\[ \frac{ |I_1(\beta)|}{|I_2(\beta)|} \leq \frac{32 \sqrt{2} \delta}{c} \]
\end{cor}
\begin{proof}
Notice that $I_2(\beta)$ contains the set 
\[I^h\left( \beta, \frac{c}{32\sqrt{2}} \right) \cap I^v \left( \beta, \frac{c}{32\sqrt{2}} \right) \]
The proof of Lemma \ref{L:interval2} with $c$ replaced by $\frac{c}{\sqrt{2}}$ shows that this intersection is exactly $I^h\left( \beta, \frac{c}{32\sqrt{2}} \right)$, which establishes the first inclusion. The second inclusion is immediate from Lemma \ref{L:interval2}. 
\end{proof}

\begin{lemma}\label{L:bad-interval-estimate}
For any interval of the form $[t, 2t] \subseteq \left( \frac{2 \log |\beta_1|}{\delta}, 2^{m} \frac{2 \log |\beta_1|}{\delta} \right)$, the subset that is contained in some $I_1(\beta)$, for some cylinder $\beta$ satisfying $\frac{|\beta_1|}{2\sqrt{2}} \leq |\beta|$, has length at most $\frac{\delta (192\sqrt{2}g - 192\sqrt{2})}{c} t$
\end{lemma}

Note our choice of $\delta$ in Assumption \ref{A:starting1} says the above quantity is at most $\frac{t}{2}$.

\begin{proof}
We will proceed in three steps.

\noindent \textbf{Step 1: Any $3g-2$ intervals of the form $I_2(\beta)$ have empty intersection}

Such an intersection would contain $(3g-2)$ cylinders that have length at most $\frac{c}{32}$ and area at least $c$, hence height at least $32$. These cylinders cannot cross each other and hence such an intersection would contradict the fact that there are at most $3g-3$ disjoint simple closed curves on a surface of genus $g$. 

\noindent \textbf{Step 2: If $I_1(\beta)$ intersects $[t, 2t]$ then $I_2(\beta)$ is no longer than $\frac{t}{4}$}

If $I_1(\beta)$ intersects $[t, 2t]$ then the $t_0$ corresponding to $\beta$ must satisfy 
\[ t_0 - \frac{4\sqrt{2}\delta t_0}{c} \leq 2t \quad \text{that is} \quad t_0 \leq \frac{2t}{1 - \frac{4\sqrt{2}\delta}{c}} \]
This shows that $t_0 \leq 2\sqrt{2} t$ and so by the simplified estimates in Remark \ref{R:size-estimate}, the radius of $I_2(\beta)$ is at most $\frac{t}{2}$. 


\noindent \textbf{Step 3: The subset of $[t, 2t]$ that is contained in some $I_1(\beta)$ has length at most $\frac{\delta (192\sqrt{2}g - 192\sqrt{2})}{c} t$}

Let $J$ be the collection of $t_0$ from cylinders $\beta$ so that $I_1(\beta)$ intersects $[t, 2t]$. For any such $\beta$,  $I_2(\beta) \subseteq [\frac{t}{2}, \frac{5t}{2}]$ and each element in the interval $[\frac{t}{2}, \frac{5t}{2}]$ may lie in at most $3g-3$ intervals of the form $I_2(\beta)$. Therefore, 
\[ \sum_{t_0\in J} |I_1(\beta)| \leq \frac{32 \sqrt{2} \delta}{c} \sum_{t_0\in J} |I_2(\beta)| \leq \frac{(192\sqrt{2}g - 192\sqrt{2})\delta}{c} t\]
where the first inequality holds by Corollary \ref{C:interval-estimate}. 
\end{proof}

\begin{cor}\label{C:lots}
For any interval of the form $[t, 2t] \subseteq \left( \frac{2 \log |\beta_1|}{\delta}, 2^{m} \frac{2 \log |\beta_1|}{\delta} \right)$, there are at least $\left( 1-\frac{\delta (192\sqrt{2}g - 192\sqrt{2})}{c} \right) t -1$ points in $[t, 2t]$ that are not contained in any $I_1(\beta)$, for some cylinder $\beta$ satisfying $\frac{|\beta_1|}{2\sqrt{2}} \leq |\beta|$, and that are separated by at least unit distance. 
\end{cor}
\begin{proof}
Let $\nu := \frac{\delta (192\sqrt{2}g - 192\sqrt{2})}{c} $. Let $p_1$ be the first point in $[t, 2t]$ not contained in some $I_1$ and set $a_1 = p_1 - t$. Let $p_2$ be the next point that lies beyond $p_1 + 1$ and set $a_2 = p_2 - (p_1 + 1)$. Iterate this procedure until $p_n$ lies within unit distance of $2t$. Let the leftover distance at the end be $\rho = 2t - p_n$. By Lemma \ref{L:bad-interval-estimate}, $\sum_{i=1}^n a_i < \nu t$. Since 
\[ n + \sum_{i=1}^n a_i + \rho = t \]
we must have
\[ n \geq t - 1 - \nu t  = (1-\nu)t - 1\]
\end{proof}

%
%

\section{Getting started} \label{S:started}

We will continue to make Assumptions \ref{A:basic1} and \ref{A:starting1} in this section. Make the following definitions. 

\begin{defn}
Let $\mathrm{Sys}(X, \omega)$ be the length of the shortest saddle connection on $(X, \omega)$. Set $T_1:= 2g+|\Sigma|-2$ and $T_0 = 2^{\left( 2^{4T_1} \right)}$. Let $\Theta(R) \subseteq \left[ \frac{-\pi}{2}, \frac{\pi}{2} \right]$ be the collection of the arguments of holonomy vectors of core curves of cylinders whose circumference is at most $R$ and whose area is at least $\frac{1}{T_1}$.
\end{defn}

We will use the following theorem about the distribution of cylinders on translation surfaces, which is based on work of Chaika \cite{Chaika11} and Vorobets \cite{Vo05}.

\begin{thm}(Marchese-Trevi\~no-Weil \cite[Theorem 1.9(4)]{MTW18})\label{T:MTW}
Fix $K \geq \frac{\sqrt{2} T_0^2}{\mathrm{Sys}(X, \omega)}$ and a positive integer $n \geq 1$. For any interval $I \subseteq \left[ \frac{-\pi}{2}, \frac{\pi}{2} \right]$ such that
\[ |I| \geq \frac{1}{2T_1\mathrm{Sys}(X, \omega) K^{n-1}} \]
at least half of the points in $I$ are within $\frac{\sqrt{3}K}{K^{2n}}$ of an element of $\Theta(R)$. 
\end{thm}

We will only use the following immediate consequence. 

\begin{cor}\label{C:MTW}
There is an $R_0'$ so that for any $R > R_0'$ there are constants $d_1,d_2,d_3$ independent of $R$ such that there are  $d_1R^2$ cylinders of circumference at most $R$, area bounded below by $d_2$, and whose angles are at least $\frac{d_3}{R^2}$ apart.
\end{cor}
\begin{proof}
Set $\lambda := \frac{\sqrt{2} T_0^2}{\mathrm{Sys}(X, \omega)}$. For any $r > \lambda$ there is some $\ell \in [\lambda, \lambda^2]$ so that $r = \ell^n$ for some positive integer $n$. Take the interval $I$ to be $\left[ \frac{-\pi}{2}, \frac{\pi}{2} \right]$. Theorem \ref{T:MTW} states that half of all points in $I$ are within $\frac{\sqrt{3}\lambda^2}{r^2}$ of an element of $\Theta(r)$ when $r > \lambda$. 

Fix $r > \lambda$ and divide the circle into intervals of equal size that are as close as possible to radius $\frac{1}{r^2}$. There will be at least $\frac{\pi}{2}r^2 - 1 > \frac{r^2}{2}$ intervals. Let $N$ be the least integer greater than or equal to $\sqrt{3} \lambda^2$. If an element of $\Theta(r)$ is contained in one of the intervals, then the ball of radius $\sqrt{3} \lambda^2$ about it is contained in $2N+1$ intervals. The ball of radius $\sqrt{3} \lambda^2$ about any point in those $2N+1$ intervals is contained in $4N+1$ intervals. 

Let $S$ be a maximal collection of points in $\Theta(r)$ that are all pairwise distance $\frac{4N+2}{r^2}$ apart. Theorem \ref{T:MTW} implies that $|S|(4N+1) \geq \frac{r^2}{4}$. Therefore, we have found $\frac{r^2}{16N+4}$ cylinders of circumference less than $r$, area at least $\frac{1}{T_1}$, and which are separated by a distance of $\frac{4N+2}{r^2}$.  
\end{proof}

\begin{rem}
We see that all constants are explicit, that is, we may take $R_0' =  \frac{\sqrt{2} T_0^2}{\mathrm{Sys}(X, \omega)}$, $d_1 = \frac{1}{16\sqrt{3}(R_0')^2 + 20}$, $d_2 = \frac{1}{2g+|\Sigma|-2}$, and $d_3 = 4\sqrt{3}(R_0')^2 + 6$.
\end{rem}

We will also use the quadratic asymptotics of cylinders,

\begin{thm}(Masur \cite[Theorem 1]{Ma-complex}) \label{T:Masur}
There is also a constant $d_4$ such that for any $R$ there are at most $d_4R^2$ cylinders of circumference at most $R$.
\end{thm}

Now make the following assumption.

\begin{ass} \label{A:starting-c} 
Set $D:= \max\left( 2, \sqrt{\frac{2d_4}{d_1}} \right)$. Suppose that $c < d_2$ and that $\delta < \frac{c}{512 D^4}$. 
\end{ass}

\begin{prop}\label{P:getting-started}
For $R$ sufficiently large, there is a cylinder $\beta$ of circumference at least $R$, area at least $c$, and that contains a protochild whose protochild surface is $(\delta, c)$-thick. 
\end{prop}
\begin{proof}
Fix $R > \max\left( R_0',  \mathrm{exp}\left( \frac{4}{d_3}\right), D, \frac{1}{\mathrm{Sys}(X, \omega)} \right)$. By Corollary \ref{C:MTW}, let $\mathrm{Cyl}$ be a collection of $d_1 (DR)^2 \geq 2d_4 R^2$ cylinders, of circumference at most $DR$, area bounded below by $d_2$, and whose angles are at least $\frac{d_3}{(DR)^2}$ apart. By Theorem~\ref{T:Masur}, there are at most $d_4 R^2$ cylinders of circumference less than $R$. Let $\cC$ be the subcollection of at least $d_4 R^2$ cylinders in $\mathrm{Cyl}$ whose circumference is in $[R, DR]$. 

\noindent \textbf{Step 1: Two distinct cylinders in $\cC$ have disjoint sets of protochildren and the sets of protochildren have length at least $\frac{\delta c}{16 (DR)^2 \log(DR)}$.}

Given a cylinder $\beta \in \cC$, the collection of protochildren has the form $[N, 2^m N]$ where $N = \frac{2 \log |\beta|}{\delta}$. Therefore, 
\[ \sin \theta\left( s_N, s_{2^m N} \right) = \frac{\mathrm{area}(\beta) \left| 2^m N - N \right|}{|s_N||s_{2^m N}|} \]
Using that $|s_N| \leq 2N |\beta|$ and a similar estimate for $|s_{2^mN}|$, it follows that
\[ \frac{\delta c}{16 |\beta|^2 \log |\beta|} \leq \sin \theta\left( s_N, s_{2^m N} \right) \leq \theta\left( s_N, s_{2^m N} \right) \]
Since $|\beta|$ is large it follows that $\sin \theta \left( s_N, \beta \right)$ is small and so, 
\[ \theta \left( s_N, \beta \right) \leq 2 \sin \theta \left( s_N, \beta \right) \leq \frac{2}{|\beta|^2 N} = \frac{\delta}{|\beta|^2 \log |\beta|} \] 
Since $s_N$ is the furthest point from $\beta$, we have that the largest distance from $\beta$ to an element of its protochild set is  $\frac{\delta}{|\beta|^2 \log |\beta|} \leq \frac{\delta}{R^2 \log R}$. If $\beta$ and $\beta'$ are two distinct cylinders in $\cC$ then they are separated by a distance of at least $\frac{d_3}{(DR)^2}$. Therefore, the distance between the two sets of protochildren of $\beta$ and $\beta'$ is at least
\[ \frac{d_3}{(DR)^2} - \frac{2\delta}{R^2 \log R}.  \]
Using the estimate that $\delta < \frac{1}{512 D^4}$ and $\log R > \frac{4}{d_3}$, we see that the distance between the two sets of protochildren is at least $\frac{d_3}{2(DR)^2}$. Hence, the sets of protochildren are disjoint.

Let $\cI$ be the collection of all protochildren of cylinders in $\cC$. Partition the interval $[1, DR]$ into subintervals $I_j := [\frac{R}{D^j},\frac{R}{D^{j-1}}]$ with $j \geq 0$. We analyze the following cases.
 
 \noindent \textbf{Step 2: Cylinders of circumference greater than $DR$ can only cause $(\delta, c)$-thinness for half of all protochildren in $\cI$.}
 
 This step follows immediately from Assumption \ref{A:starting1} and Lemma \ref{L:bad-interval-estimate}.



 \noindent \textbf{Step 3: Cylinders of circumference less than $DR$ cause $(\delta, c)$-thinness for at most a quarter of all protochildren in $\cI$.}

 Suppose that $\beta'$ is a cylinder whose circumference belongs to $I_j$ and which is thin for the protochild surface of $s_t$ of $\beta$. Then 
\[ |\beta'| \sin \theta\left( \beta', s_t \right) |s_t| \leq \delta \]
In other words,
\[ \sin \theta\left( \beta', s_t \right)  \leq \frac{\delta}{|s_t||\beta'|} \leq \frac{\delta^2 D^j}{2R^2 \log R} \]
Notice that we may assume that $D^j < R^2$ because $\frac{1}{R} < \mathrm{Sys}(X, \omega)$. This implies that $\sin \theta\left( \beta', s_t \right) < \frac{\delta^2}{2 \log R}$, i.e. the sine of the angle is so small that we can use the estimate $\frac{\theta}{2} < \sin \theta$. Therefore, the length of the collection of angles for which $\beta'$ is $(\delta, c)$-thin on the corresponding protochild surface is at most 
\[ \frac{2 \delta^2 D^j}{R^2 \log R} \]

By Theorem \ref{T:Masur}, there are at most $\frac{d_4 R^2}{D^{2j-2}}$ cylinders with circumference in $I_j$. Since these cylinders are $(\delta, c)$-thin on protochild surfaces corresponding to an interval of angles of length at most $\frac{2 \delta^2 D^j}{R^2 \log R}$, the total length of the collection of angles for which a cylinder with circumference in $I_j$ is $(\delta, c)$-thin is at most
\[  \left( \frac{d_4 R^2}{D^{2j-2}} \right) \left( \frac{2 \delta^2 D^j}{R^2 \log R} \right) = \left( \frac{2 d_4 D^2 \delta^2}{\log R} \right) \frac{1}{D^j} \]

However, there are at least $d_4 R^2$ cylinders in $\cC$ and each has a collection of protochildren that has length at least $\frac{\delta c}{16 (DR)^2 \log(DR)}$. Therefore, the length of the collection of protochildren associated to cylinders in $\cC$ is at least
\[ \left( d_4 R^2 \right) \left( \frac{\delta c}{16 (DR)^2 \log(DR)}\right) = \left( \frac{\delta c d_4}{16 D^2 \log(DR) }\right)  \]
Therefore, the largest proportion of $\cI$ whose protochild surface contains a $(\delta, c)$-thin cylinder that has  circumference smaller than $DR$ on $(X, \omega)$ is 
\[ \frac{ \left( \frac{2 d_4 D^2 \delta^2}{\log R} \right) }{\left( \frac{\delta c d_4}{16 D^2 \log(DR) }\right)} \sum_{j=0}^\infty \frac{1}{D^j} < \left( \frac{32 \delta D^4}{c} \right) \left( \frac{\log(DR)}{\log(R)} \right) \left( \frac{D}{D-1} \right)  \]
Since $2 \leq D \leq R$, this ratio is bounded above by $\left( \frac{128 D^4 \delta}{c} \right)$, which is at most $\frac{1}{4}$ by Assumption \ref{A:starting-c}. Combining these steps we conclude that a fourth of all protochildren have protochild surfaces that are $(\delta, c)$-thick.

\end{proof}

\section{Elementary Facts about the child selection process}\label{S:facts}

Throughout this section Assumption \ref{A:basic1} will continue to hold and we will make the following assumption.

\begin{ass}\label{A1}
Suppose that $\cot \theta_1 < \frac{c}{16}$. Let $\beta_0$, $\beta_1$, and $s_t$ be defined as in definition of the child selection process (Definition \ref{D:CS}). Let $C:= \frac{Lc}{16M}$. Suppose that $|\beta_0| > \max \left( e^{4/C}, 21, e^{2\delta}, e^M \right)$. This implies that all estimates in this section (including the ones contingent on the size of $\beta_0$) hold.
\end{ass} 

\begin{lemma}[Length Facts]\label{L:length-facts}
The following facts hold for the child selection process,
\begin{enumerate}
\item \label{F:L1} If $|\beta_0|>1$, then it is immediate that $$ t |\beta_0| \leq |s_t| \leq (t+1) |\beta_0|.$$ Since $t \in \left(\frac{2 \log|\beta_0|}{\delta},  \frac{M \log|\beta_0|}{2L} \right)$, it follows that $$\frac{2}{\delta} |\beta_0| \log|\beta_0| \leq |s_t| \leq \frac{M}{L} |\beta_0| \log|\beta_0|.$$
If $e^{\delta/2} < |\beta_0|$, then $|\beta_0| < |s_t|$.  
\item \label{F:L2}  Since the length of $\beta_1$ on $(X_1, \omega_1)$ is between $\delta$ and $L$ and since it makes an angle of at least $\frac{\pi}{4}$ with the horizontal, when $\beta_1$ is pulled back to $(X_0, \omega_0)$ by $g_{-|s_t|}$ we have $$|s_t| \frac{\delta}{2} \leq |\beta_1| \leq |s_t|L$$ Combined with the previous estimate this yields $$|\beta_0| \log|\beta_0| \leq |\beta_1| \leq M |\beta_0| \log |\beta_0|$$
\item \label{F:L3} If $|\beta_0| > e^M$, 
\[ \frac{|\beta_0| \log |\beta_0|}{ |\beta_1| \log |\beta_1|} \geq \frac{1}{6Lt} \]
\end{enumerate}
\end{lemma}
\begin{proof}
Only the proof of (\ref{F:L3}) remains to be given. First, 
\[ \frac{|\beta_0|}{|\beta_1|} \geq \frac{|\beta_0|}{L|s_t|} \geq \frac{|\beta_0|}{2Lt|\beta_0|} = \frac{1}{2Lt} \]
where the first inequality is from (\ref{F:L2}) and the second is from (\ref{F:L1}). Finally, by (\ref{F:L2}) we have
\[ \frac{\log |\beta_1|}{\log |\beta_0|} \leq \frac{M}{\log |\beta_0|} + \frac{\log |\beta_0|}{\log |\beta_0|}  + \frac{\log \log |\beta_0|}{\log |\beta_0|} \leq 3 \]
where the final inequality comes from the fact that each summand is less than $1$. 
\end{proof}

\begin{lemma}[Angle Facts 1]\label{L:angle-facts1}
The following facts hold for the child selection process,
\begin{enumerate}
\item \label{F:A1F1} Since $|\beta_0 \times s_t| = \mathrm{area}(\beta_0)$, 
\[\frac{c}{|\beta_0||s_t |} \leq  \sin \theta(\beta_0, s_t) = \frac{\mathrm{area}(\beta_0)}{|\beta_0||s_t|} \leq \frac{1}{|\beta_0||s_t|} \]
\item \label{F:A1F2} The largest angle that $\beta_1$ makes with the vertical is when it lies in the direction of $(\cos \theta_1, \sin \theta_1)$, which pulls back to $(\frac{\cos \theta_1}{|s_t|}, |s_t| \sin \theta_1)$ on $(X_0, \omega_0)$. Therefore, $$ \left| \tan \theta(\beta_1, s_t) \right| \leq \frac{\cot \theta_1}{|s_t|^2} \leq \frac{(\delta/2)^2}{|\beta_0|^2 \log^2 |\beta_0|}.$$
\item \label{F:A1F3} Since $$|s_t \times s_{t'}| = \left| \left( s + t \beta_0 \right) \times \left( s + t' \beta_0 \right) \right| = \mathrm{area}(\beta_0)|t-t'|$$ it follows that $$ \left| \sin \theta(s_t, s_{t'}) \right| = \frac{\mathrm{area}(\beta_0) |t - t'|}{|s_t| |s_{t'}|}.$$ If additionally, $t,t'\in \left[ \frac{2 \log|\beta_0|}{\delta} , \frac{M \log|\beta_0|}{2L} \right]$ and $|t-t'|\geq 1$, then $$ \frac{c(2L/M)^2}{|\beta_0|^2 \log^2 |\beta_0|} \leq \left| \sin \theta(s_t, s_{t'}) \right|.$$ 
\item \label{F:A1F4} Suppose that $\beta'$ and  $\beta''$ are children of $\beta_0$ corresponding to $s_{t'}$ and $s_{t''}$ respectively. Suppose   $|t''-t'|\geq 1$.   Suppose also  $|\beta_0| > e^{2\delta}$. Then  
\[ \frac{c|t''-t'|}{16|\beta_0|^2|t't''|} \leq \theta(\beta', \beta'') \]
\end{enumerate}
\end{lemma}
\begin{proof}
Only the proof of (\ref{F:A1F4}) remains to be given. By the triangle inequality, 
\[ \theta(\beta'', \beta') \geq \theta(s_t'', s_t') - \theta(s_t'', \beta'') - \theta(s_t', \beta') \]
Since $\sin \theta \leq \theta \leq \tan \theta$ for all $0 \leq \theta \leq \frac{\pi}{2}$, we have
\[  \theta(\beta'', \beta') \geq \sin \theta(s_t'', s_t') - \tan \theta(s_t'', \beta'') - \tan \theta(s_t', \beta') \]
By (\ref{F:A1F2}) and (\ref{F:A1F3}) we have
\[ \theta(\beta'', \beta') \geq  \frac{\mathrm{area}(\beta_0) |t'' -t'|}{|s_t''| |s_t'|} - \frac{\cot \theta_1}{|s_t'|^2} - \frac{\cot \theta_1}{|s_t''|^2}\]
Assume without loss of generality that $t' > t''$. Using the estimate that $\cot \theta_1 < \frac{c}{16}$, 
\[ \theta(\beta'', \beta') \geq \frac{c}{4|\beta_0|^2} \left( \frac{t'-t''}{t''t'} - \frac{(1/4)}{t'^2} - \frac{(1/4)}{t''^2} \right)\]
Now since $t'\geq t''+1$, 
\[ 4(t'-t'')t't'' - t'^2 - t''^2 \geq 2(t'-t'')t''t' - (t'-t'')(t'+t'') \]
The right hand side is equal to
\[ (t'-t'')(2t''t' - t' - t'') = (t'-t'')\left( t''(t'-1) + t'(t''-1) \right) \]
Therefore, we have 
\[ \theta(\beta'', \beta') \geq \frac{c|t'-t''|}{16|\beta_0|^2t''t'} \left( \frac{t'-1}{t'} +  \frac{t''-1}{t''} \right)\]
Since $|\beta_0| > e^{2\delta}$, it follows that $t'$ and $t''$ are are greater than $1$, hence
\[ \theta(\beta'', \beta') \geq \frac{c|t'-t''|}{16|\beta_0|^2t''t'} \]
\end{proof}

\begin{defn}
Given a cylinder $\beta$, let $\theta_\beta$ be the angle it makes with the horizontal, let $I_\beta$ be the interval of angles centered at $\theta_\beta$ with radius $\frac{1}{|\beta|^2 \log |\beta|}$. Define $N_\beta := \log |\beta|$ and  $\rho_{\beta} := \frac{C}{\log |\beta|}$.
\end{defn}

\begin{lemma}[Angle Facts 2]\label{L:angle-facts2}
Suppose that $|\beta_0| > \max(21, e^{4/C})$. The following facts hold for the child selection process,
\begin{enumerate}
\item \label{L:AF2F1} Suppose $\beta_1$ is a child of $\beta_0$. Then $I_{\beta_1} \subseteq I_{\beta_0}$
\item \label{L:AF2F2} Suppose that $\beta'$ and $\beta''$ are distinct children of $\beta_0$ corresponding to $s_{t'}$ and $s_{t''}$ respectively and suppose that $|t' - t''| \geq 1$. Then the distance between $I_{\beta'}$ and $I_{\beta''}$ is at least $\rho_{\beta_0} |I_{\beta_0}|$.
\end{enumerate}
\end{lemma}
\begin{proof}
(Proof of (\ref{L:AF2F1})) The radius of $I_{\beta_1}$ is $\frac{1}{|\beta_1|^2 \log |\beta_1|}$ and so the largest angle between an element of $I_{\beta_1}$ and $\beta_0$ is bounded by
\[ \frac{1}{|\beta_1|^2 \log |\beta_1|} + \theta(s_t, \beta_0) + \theta(s_t, \beta_1) \]  for any $s_t$.
Since $|\beta_0| > 4$, it follows that from Lemma \ref{L:angle-facts1} (\ref{F:A1F1}) that $\sin \theta(\beta_0, s_t) \leq \frac{1}{4}$ and hence that $\theta(\beta_0, s_t) \leq 2 \sin \theta(\beta_0, s_t)$. From this observation and from Lemma \ref{L:angle-facts1} (\ref{F:A1F2}), we have that the largest angle between an element of $I_{\beta_1}$ and $\beta_0$ is at most
\[ \frac{1}{|\beta_1|^2 \log |\beta_1|} + \frac{2}{|\beta_0||s_t|} + \frac{(\delta/2)^2}{|\beta_0|^2 \log^2 |\beta_0|} \]
By Lemma \ref{L:length-facts} this is at most
\[ \frac{1}{|\beta_0|^2 \log^2 |\beta_0|} + \frac{\delta}{|\beta_0|^2 \log |\beta_0| } + \frac{(\delta/2)^2}{|\beta_0|^2 \log^2 |\beta_0|} \]
Since $|\beta_0| > 21$, $\frac{1}{\log |\beta_0|} < \frac{1}{3}$. Since $\delta < \frac{1}{3}$ and so the largest angle between an element of $I_{\beta_1}$ and $\beta_0$ is less than $\frac{1}{|\beta_0|^2 \log |\beta_0|}$ as desired.

(Proof of (\ref{L:AF2F2})) By definition, the interval $I_{\beta'}$ with center $\theta_{\beta'}$ has radius $\frac{1}{|\beta'|^2\log |\beta'|}$ whereas, since $|t' - t''| \geq 1$, we have by Lemma \ref{L:angle-facts1} (\ref{F:A1F4}) that $\theta(\beta',\beta'')\geq \frac{c|t'-t''|}{16|\beta_0|^2t't''}$.  The distance between two distinct intervals $I_{\beta'}$ and $I_{\beta''}$ is at least
\[ \frac{1}{|\beta_0|^2 \log^2 |\beta_0|} \left(\frac{Lc}{8M} - \frac{2}{\log |\beta_0| } \right)  \]
\noindent Since $|\beta_0| > e^{4/C}$, the distance is at least 
$\rho_{\beta_0}|I_{\beta_0}|$ 
as desired.

\end{proof}

%
%

\section{Cylinders that cause thinness are comparable in size to parent cylinders} \label{S:comparable}


We continue to assume Assumptions \ref{A:basic1} and \ref{A1} and keep the notation of Definition \ref{D:CS}. The main result of this section is the following 

\begin{prop} \label{P:comparable}
For sufficiently large $|\beta_0|$, if $\sigma_1$ is a protochild of $\beta_1$ whose protochild surface has a $(\delta,c)$-thin cylinder $\beta_2$ then on $(X_0, \omega_0)$ we have \[ \frac{|\beta_1|}{2\sqrt{2}} \leq |\beta_2|. \]
\end{prop}

 We make the following assumption for this section. 

\begin{ass}\label{A2}
Let $\sigma_1$ be a protochild of $\beta_1$ whose protochild surface $(X_2, \omega_2)$ has a $(\delta, c)$-thin cylinder $\beta_2$. Rename the protochild of $\beta_0$ to $\sigma_0$ and suppose without loss of generality, that it is vertical on $(X_0, \omega_0)$. 
\end{ass} 

\begin{rem}
Note here that the subscripts $0,1,$ and $2$ do not refer to times  but to labeling. Also, recall the convention that all angles and lengths will be measured on the $(X_0, \omega_0)$ unless otherwise mentioned. 
\end{rem}

\begin{lemma}\label{L:cs-1}
For sufficiently long $\beta_0$, there is some constant $c_3$ depending only on the constants in Assumption \ref{A1} so that the angle $\phi$ between $\beta_1$ and $\beta_2$ satisfies 
\[ \phi \leq \frac{c_3}{|\beta_1|^2 \log |\beta_1|} \]
\end{lemma}
\begin{proof} 
Let $\theta=\theta(\sigma_0,\sigma_1)$ be the angle between $\sigma_0$ and $\sigma_1$. 
Let $\theta'$ be the angle that the holonomy vector of $\beta_2$ makes with the horizontal on $(X_1, \omega_1)$. We proceed in three steps. 

\noindent \textbf{Step 1: For $\log|\beta_0| \geq \frac{4}{\delta\cot\theta_1}$, $\theta(\sigma_0, \sigma_1) \leq \frac{3\cot \theta_1}{|\sigma_0|^2}$.}

\noindent By Lemma \ref{L:angle-facts1} (\ref{F:A1F2}), 
\[ \theta(\beta_1, \sigma_0) \leq \tan \theta(\beta_1, \sigma_0) \leq \frac{\cot \theta_1}{|\sigma_0|^2}. \]
Similarly, by Lemma \ref{L:angle-facts1} (\ref{F:A1F1}) and that fact that $\theta(\beta_1, \sigma_1)$ is less than $\frac{\pi}{2}$, 
\[ \theta(\beta_1, \sigma_1) \leq 2 \sin \theta(\beta_1, \sigma_1) \leq \frac{2}{|\sigma_1||\beta_1|}. \]
The triangle inequality now implies that
\[ \theta=\theta(\sigma_0, \sigma_1) \leq \frac{2\cot \theta_1}{|\sigma_0|^2} + \frac{2}{|\sigma_1||\beta_1|} = \frac{1}{|\sigma_0|^2} \left( 2 \cot \theta_1 + 2 \frac{|\sigma_0|}{|\sigma_1|}  \frac{|\sigma_0|}{|\beta_1|}  \right). \] 

Again by Lemma~\ref{L:length-facts}, the ratio $ \frac{|\sigma_0|}{|\beta_1|}\leq \frac{2}{\delta}$  and the ratio $\frac{|\sigma_0|}{|\sigma_1|}\leq \frac{1}{\log |\beta_0|}$ so our choice of $|\beta_0|$  gives 
\[ \theta \leq \frac{3\cot \theta_1}{|\sigma_0|^2} \]


\noindent \textbf{Step 2: If $|\beta_0|\geq\sqrt{\frac{12L^2\cot\theta_1}{\pi}}$ and $\log |\beta_0|\geq  \frac{1}{\cot \theta_1}$ then  $|\cot \theta'| \leq 10 \cot \theta_1$}.

 The matrix that passes from $(X_1, \omega_1)$ to $(X_2, \omega_2)$ is 
\[ g:= g_{\log |\sigma_1|} r_\theta g_{-\log |\sigma_0|} = \begin{pmatrix} \frac{|\sigma_1|}{|\sigma_0|} \cos \theta &  -|\sigma_0||\sigma_1|\sin \theta \\ \frac{\sin \theta}{|\sigma_0||\sigma_1|} & \frac{|\sigma_0|}{|\sigma_1|} \cos \theta \end{pmatrix} \]

\noindent Therefore, if some vector $(h, v)$ has length less than $\delta$ after applying $g$, it follows that
\[ \left| h\frac{|\sigma_1|}{|\sigma_0|} \cos \theta - v |\sigma_0||\sigma_1|\sin \theta \right| \leq \delta \]
By the triangle inequality,
\[ |h| \frac{|\sigma_1|}{|\sigma_0|} |\cos \theta| \leq |v| |\sigma_0||\sigma_1| |\sin \theta| + \delta \]
Since $|\sigma_0| > \frac{|\beta_0|}{L}$, Step 1 implies 
$$\theta\leq \frac {3L^2\cot \theta_1}{|\beta_0|^2}$$
so that by our choice of $|\beta_0|$ we have $\cos \theta > \frac{1}{2}$. This implies that 
\[|h| \leq 2 \delta  \frac{|\sigma_0|}{|\sigma_1|}  + 2 |v| |\sigma_0|^2 \theta \]
 
By Step 1 we have
\[ |h| \leq 2 \delta  \frac{|\sigma_0|}{|\sigma_1|}  + 6 \cot \theta_1 |v| \]

Assume now that $(h, v)$ is the holonomy of $\beta_2$ on $(X_1, \omega_1)$. Since $(X_1, \omega_1)$ is $(\delta, c)$-thick it implies that $(h, v)$ has length at least $\delta$. Since $\frac{|\sigma_0|}{|\sigma_1|}\leq \frac{1}{|\log |\beta_0|}$ it follows  that $\frac{|\sigma_0|}{|\sigma_1|} < \cot \theta_1<\frac{1}{16}$. Notice that 
$|v|>\frac{\delta}{2}$ 
since otherwise by the above bound on $|h|$  both $|v|$ and $|h|$ would be smaller than $\frac{\delta}{2}$ (and hence $(h, v)$ is not longer than $\delta$). Then $2\delta<4|v|$.
Combined with the previous estimate we see that 
\[ \cot \theta'=\frac{|h|}{|v|} < 10 \cot \theta_1 \]

\noindent \textbf{Step 3: For  $|\beta_0|^2\geq \max(\frac{4L}{\pi\sin\theta_2},\frac{8}{\delta \pi})$ there is some constant $c_3$ depending only on the constants in Assumption \ref{A1} so that $\phi \leq \frac{c_3}{|\beta_1|^2 \log |\beta_1|}$.} 

Define $\theta_2 := \mathrm{arccot}(10 \cot \theta_1)$. By Step 2, the vertical part of the holonomy of $\beta_2$ on $(X_1, \omega_1)$ is at least $\delta \sin \theta_2$. Therefore, on $(X_0, \omega_0)$,
\[ |\beta_2| \geq \delta |\sigma_0| \sin \theta_2 \]
Let $\alpha_1$ be the angle between $\sigma_1$ and $\beta_2$. Let $(X_0', \omega_0')$ be the surface $(X_0, \omega_0)$ rotated so that $\sigma_1$ is vertical. On this surface the horizontal part of the holonomy of $\beta_2$ is at most $\frac{\delta}{|\sigma_1|}$ since $\beta_2$ has length less than $\delta$ on $(X_2, \omega_2)$. Therefore, by the above lower bound on $|\beta_2|$
\[ \sin \alpha_1 \leq \frac{\delta}{|\beta_2||\sigma_1|} \leq \frac{1}{\sin \theta_2 |\sigma_0||\sigma_1|}\leq \frac{L}{\sin \theta_2 |\beta_0|^2\log |\beta_0|} \]
Let $\alpha_2$ be the angle between $\sigma_1$ and $\beta_1$.  We have  
\[ \sin \alpha_2 \leq \frac{1}{|\beta_1||\sigma_1|}\leq \frac{2}{\delta |\beta_0|^2(\log |\beta_0|)^2}. \]
 Now are choice of $|\beta_0|$ says $$\frac{\alpha_i}{2} \leq \sin \alpha_i,$$ for $i \in \{1, 2\}$ 

By the triangle inequality, 
\[ \phi \leq 2  \left(\frac{1}{\sin \theta_2 |\sigma_0||\sigma_1|}+\frac{1}{|\beta_1||\sigma_1|}\right) \]

By Lemma \ref{L:length-facts}
 $|\beta_1|$ is comparable to $|\sigma_0|$ and $|\sigma_1|\geq |\beta_1| \log |\beta_1|$.
It follows that there is a constant $c_3$ only depending on the constants in Assumption~\ref{A1} 
\[ \phi \leq \frac{c_3}{|\beta_1|^2 \log |\beta_1|} \]

\end{proof}

\begin{proof}[Proof of Proposition \ref{P:comparable}:]
Suppose to a contradiction that $\frac{|\beta_1|}{2\sqrt{2}} \geq |\beta_2|$. The proof will be divided into two steps:

\noindent \textbf{Step 1: For $\log|\beta_0|\geq \frac{c_3}{2\sqrt{2} \delta^2\sin\theta_0}$, on $(X_1,\omega_1)$ we have $|\beta_2|\geq |\beta_1|$.}

Let $\alpha_i$ be the angle between $\beta_1$ and $\beta_2$ on $(X_i, \omega_i)$ for $i \in \{0, 1\}$. Let $|\cdot|_i$ denote length on $(X_i, \omega_i)$ for $i \in \{0, 1\}$. We have
\[ \delta^2 \sin \alpha_1 \leq |\beta_1|_1 |\beta_2|_1 \sin \alpha_1 = |\beta_1| |\beta_2| \sin \alpha_0 \leq \frac{c_3 |\beta_2|}{|\beta_1| \log |\beta_1|} \]
where the lefthand inequality follows from the fact that $(X_1, \omega_1)$ is $(\delta, c_1)$-thick and the righthand inequality follows from Lemma \ref{L:cs-1}. By assumption, we have
\[ \sin \alpha_1 \leq \frac{c_3 }{2\sqrt{2} \delta^2 \log |\beta_1|}. \]
By our choice of $|\beta_0|$, since $|\beta_1| > |\beta_0|$ it follows that  $\sin \alpha_1 < \sin \theta_0$. However, if $|\beta_2|\leq |\beta_1|$ on $(X_1, \omega_1)$ then by Assumption \ref{A1}, the angle between $\beta_1$ and $\beta_2$ on $(X_1, \omega_1)$ is bounded below by $\theta_0$, which is a contradiction. 

\noindent \textbf{Step 2: For $|\beta_0|\geq \frac{2c_3}{\delta}$ on $(X_0, \omega_0)$  $|\beta_2|\geq \frac{|\beta_1|}{2\sqrt{2}}$ .}

Suppose not to a contradiction. The holonomy vector of the core curve of $\beta_2$ on $(X_0, \omega_0)$ is 
\[ \left( |\beta_2| \cos \left( \varphi + \alpha_0 \right), |\beta_2| \sin \left( \varphi + \alpha_0 \right) \right) \]
where $\varphi$ is the angle that the holonomy vector of $\beta_1$ on $(X_0, \omega_0)$ makes with the horizontal. The holonomy of $\beta_2$ on $(X_1, \omega_1)$ is 
\[ \left( |\sigma_0| |\beta_2| \cos \left( \varphi + \alpha_0 \right), \frac{|\beta_2|}{|\sigma_0|} \sin \left( \varphi + \alpha_0 \right) \right) \]
Therefore we have (using the comparison of $L^1$ and $L^2$ norms on $\R^2$ - $\frac{1}{\sqrt{2}} \|\cdot \|_1 \leq \| \cdot \|_2 \leq \| \cdot \|_1$), 
\[ \frac{1}{\sqrt{2}} \left( |\sigma_0| |\beta_1| \cos \left( \varphi \right) + \frac{|\beta_1|}{|\sigma_0|} \sin \left( \varphi  \right) \right) \leq |\sigma_0| |\beta_2| \cos \left( \varphi + \alpha_0 \right) + \frac{|\beta_2|}{|\sigma_0|} \sin \left( \varphi + \alpha_0 \right)  \]
The right hand side is bounded above by
\[ |\sigma_0||\beta_2| \cos \varphi + |\sigma_0||\beta_2| \sin \alpha_0 +  \frac{|\beta_2|}{|\sigma_0|} \sin \varphi +  \frac{|\beta_2|}{|\sigma_0|} \sin \alpha_0 \]
Subtracting $\frac{1}{2\sqrt{2}} \left( |\sigma_0| |\beta_1| \cos \left( \varphi \right) + \frac{|\beta_1|}{|\sigma_0|} \sin \left( \varphi  \right) \right)$ from both sides of the inequality yields (along with the estimate $|\beta_2| \leq \frac{|\beta_1|}{2\sqrt{2}}$), 
\[ \frac{1}{2\sqrt{2}} \left( |\sigma_0| |\beta_1| \cos \left( \varphi \right) + \frac{|\beta_1|}{|\sigma_0|} \sin \left( \varphi  \right) \right) \leq |\sigma_0||\beta_2| \sin \alpha_0 +   \frac{|\beta_2|}{|\sigma_0|} \sin \alpha_0  \]
Again using the comparison of the norms $\| \cdot \|_1$ and $\| \cdot \|_2$ on $\R^2$ and the fact that $\beta_1$ has length at least $\delta$ on $(X_1, \omega_1)$ we have
\[ \frac{\delta}{2\sqrt{2}} \leq |\beta_2| \sin \alpha_0 +   \frac{|\beta_2|}{|\sigma_0|^2} \sin \alpha_0  \]
It now follows from Lemma \ref{L:cs-1} that
\[ \frac{\delta}{2\sqrt{2}} \leq \frac{c_3 |\beta_2| }{|\beta_1|^2 \log |\beta_1|} +   \frac{c_3 |\beta_2|}{|\sigma_0|^2 |\beta_1|^2 \log |\beta_1|}  \]
Applying the estimate $|\beta_2| \leq \frac{|\beta_1|}{2\sqrt{2}}$, 
\[ \delta \leq \frac{c_3 }{|\beta_1| \log |\beta_1|} +   \frac{c_3 }{|\sigma_0|^2 |\beta_1| \log |\beta_1|} \]
By our condition on $|\beta_0|$ both terms on the right are smaller than $\frac{\delta}{2}$ which yields a contradiction. 
\end{proof}


%
%

\section{Proof of Theorem~\ref{T1}}\label{S:T1}

In this section we give a proof of Theorem~\ref{T1}. The strategy is essentially found in Cheung~\cite[pages 23-24]{Cheung2}, in the context of continued fractions. We choose constants as well as a cylinder $\beta_0$ in the following way. Fix $(X, \omega)$. 

\begin{enumerate}
\item Let $d_1, d_2, d_3, d_4$ be the constants associated to $(X, \omega)$ as in Section \ref{S:started}.
\item Choose $c < \min \left( d_2, \frac{1}{2g(2g+|\Sigma|-2)} \right)$. Being less than the second quantity implies that we may choose $c = c_1 = c_2$ where $c_1$ and $c_2$ satisfy the conditions in Lemma \ref{L:find-cylinder}. Being less than the first quantity is required to satisfy Assumption \ref{A:starting-c}.  
\item Define $D:= \max \left( 2, \sqrt{\frac{2d_4}{d_1}} \right)$.
\item Choose $\delta < \min\left( \frac{c}{512D^4}, \frac{c}{576\sqrt{2}(g-1)}, \mathrm{Sys}(X, \omega) \right)$. Being less than the first quantity means that Assumption \ref{A:starting-c} is now completely satisfied. Being less than the final two quantities is required to satisfy Assumption \ref{A:starting1}. 
\item Choose $\theta_1$ so that $\cot \theta_1 < \frac{c}{16}$; this is required to satisfy Assumption \ref{A1}. 
\item Let $L$ and $\theta_0$ be as in Proposition \ref{P:theta}. Assumption \ref{A:basic1} is now completely satisfied. 
\item Set $\nu := \frac{\delta (192\sqrt{2}g - 192\sqrt{2})}{c}$.
\item Choose $M = \frac{2^{m+2} L}{\delta}$ for a positive integer $m$ such that $m > 6L \left( \log \left( \frac{2}{1+2\nu} \right) \right)^{-1}$ and so that $M >  21$. Assumption \ref{A:basic1} and Assumption \ref{A:starting1} are now completely satisfied. 
\item Define $T_1:= 2g+|\Sigma|-2$, $T_0 = 2^{\left( 2^{4T_1} \right)}$, $R_0' =  \frac{\sqrt{2} T_0^2}{\mathrm{Sys}(X, \omega)}$, $\theta_2 = \mathrm{arccot}\left( 10 \cot \theta_1 \right)$, and $C:= \frac{Lc}{16M}$.
\item Define 
\[ \begin{array}{c}
R_0'' := \max \bigg( R_0',  \mathrm{exp}\left( \frac{4}{d_3}\right), D, \frac{1}{\mathrm{Sys}(X, \omega)}, \mathrm{exp}\left( \frac{4}{C} \right), e^M, \mathrm{exp}\left( \frac{4}{\delta\cot\theta_1} \right),  \\
 \qquad \qquad \sqrt{\frac{12L^2\cot\theta_1}{\pi}},  \mathrm{exp} \left( \frac{1}{\cot \theta_1} \right), \sqrt{\frac{4L}{\pi\sin\theta_2}}, \sqrt{\frac{8}{\delta \pi}}\bigg) 
\end{array}\] 
\item For cylinders of circumference at least $R_0''$, Lemma \ref{L:cs-1} produces a constant $c_3$. 
\item Set $R_0 := \max \left( R_0'', \exp \left( \frac{c_3}{2\sqrt{2} \delta^2\sin\theta_0} \right), \frac{2c_3}{\delta} \right)$
\item By Proposition \ref{P:getting-started}, there is a cylinder $\beta_0$ on $(X, \omega)$ whose circumference is at least $R_0$, area at least $c$, and which contains a protochild $\sigma_0$ whose protochild surface $(X_1, \omega_1)$ is $(\delta, c)$-thick. This means that Assumption \ref{A1} is now completely satisfied. Let $\beta_1$ be the child cylinder chosen in the child-selection process (Definition \ref{D:CS}). 
\end{enumerate}

We will associate a collection of children to $\beta_1$. To each child cylinder constructed in this way we will associate a new collection of child cylinders and so on. We describe this iterative process. Let $\beta$ be a cylinder constructed in this process. Define its collection of child cylinders $D_\beta$ as follows.

Consider the set of protochildren of $\beta$, indexed by $\left( \frac{2 \log |\beta|}{\delta}, 2^m \frac{2 \log |\beta|}{\delta} \right)$. By Proposition \ref{P:comparable} any cylinder that is responsible for $(\delta, c)$-thinness of a protochild surface has circumference of size at least $\frac{|\beta|}{2\sqrt{2}}$. Divide the set of protochildren into sets
\[ I_k := \left( 2^k \frac{2 \log |\beta|}{\delta}, 2^{k+1} \frac{2 \log |\beta|}{\delta} \right) \]
for $k \in \{0, \hdots, m-1\}$. By Corollary \ref{C:lots}, there are at least $( 1- \nu) 2^k \frac{2 \log |\beta|}{\delta} -1$ points, call them $J_k'$, in $I_k$ that are unit distance apart and whose corresponding protochild surface is $(\delta, c)$-thick. Let $J_k$ be the subcollection of $J_k'$ with the largest and smallest points deleted. This is done so that any two distinct points in $J_\beta := \bigcup_{k=0}^{m-1} J_k$ are unit distance apart. Notice that
\[ |J_k| \geq ( 1- \nu) 2^k \frac{2 \log |\beta|}{\delta} - 3 > (1 - 2 \nu) 2^k \frac{2 \log |\beta|}{\delta}  \]
The set of children $D_\beta$ will then be the children constructed in the child-selection process whose indices correspond to the indices in $J_\beta$. We now summarize properties of the children constructed in this process.

\begin{prop}\label{L1}
Notice that if $\beta'$ and $\beta''$ are distinct children of $\beta$ corresponding to indices $t'$ and $t''$, then 
\begin{enumerate}
\item By Lemma \ref{L:length-facts} (\ref{F:L2}), $|\beta| \log |\beta| \leq |\beta'| \leq M |\beta| \log |\beta|$.
\item By Lemma \ref{L:angle-facts2} (\ref{L:AF2F1}), $I_{\beta'} \subseteq I_\beta$.
\item By Lemma \ref{L:angle-facts2} (\ref{L:AF2F2}), the distance between $I_{\beta'}$ and $I_{\beta''}$  is at least $\rho_\beta |I_\beta| = \frac{C}{N_\beta^2 |\beta|^2}$.
\item By Lemma \ref{L:length-facts} (\ref{F:L3}), $\frac{N_\beta|\beta|}{N_{\beta'}|\beta'|} \geq \frac{1}{6Lt'}$.
\end{enumerate}
\end{prop}

Notice that if $(\beta_n)_{n \geq 0}$ is a sequence of cylinders constructed in the above process so that $\beta_n \in D_{\beta_{n-1}}$ then $(I_{\beta_n})_{n \geq 0}$ is a nested sequence of intervals whose diameter is tending to zero. By the nested interval theorem there is an angle $\theta$ so that $\bigcap_n I_{\beta_n} = \{ \theta \}$. Let $\cD$ be the collection of angles that can be written this way. 

By Cheung \cite[Theorem 3.3]{Cheung2}, given a set $\cD$ constructed in the previously described way and satisfying the four enumerated conditions above, if $s$ is some real number so that for every cylinder $\beta$ constructed in the above process
\[ \sum_{\beta' \in D_\beta} \frac{\rho_{\beta'}^s |I_{\beta'}|^s}{\rho_\beta^s |I_\beta|^s} > 1\]
then the Hausdorff dimension of $\cD$ is at least $s$. 

To prove Theorem \ref{T1} it remains to show that $\cD$ is contained in the set of divergent on average directions and that the above inequality holds for $s = \frac{1}{2}$.

%

\begin{lemma}\label{L2}
Let $0 < \epsilon < 1$. Suppose $\beta'$ is a child of $\beta$ and $\theta\in I_{\beta'}\subset I_\beta$. Suppose too that $\frac{4}{\epsilon^2} < \log |\beta|$. Then for all $$t\in  [\log |\beta|,\log|\beta'|],$$ except for a subset of size at most $\log \frac{4M}{\epsilon^2}$, $\beta$ has flat length at most $\epsilon$ on $g_tr_\theta(X,\omega)$. 
\end{lemma}
\begin{proof} 
Suppose without loss of generality that we have rotated $(X, \omega)$ so that $\theta$ is the vertical direction.  Let $h(t)$ and $v(t)$ be the horizontal (resp. vertical) component of the period of $\beta$ on $g_t r_\theta (X, \omega)$. When $t = \log |\beta| + \log\left(\frac{2}{\epsilon} \right)$ we see that 
\[ v(t) \leq \frac{\epsilon}{2} \quad \text{and} \quad h(t) \leq \left( \frac{2|\beta|}{\epsilon} \right) |\beta| \sin \left( \frac{1}{|\beta|^2 \log |\beta|} \right) \leq \frac{2\epsilon^{-1}}{\log |\beta|} \leq \frac{\epsilon}{2}. \]
Similarly  when $t = \log |\beta'| - \log\left(\frac{2M}{\epsilon} \right)$ we see that
\[ v(t) \leq \frac{\epsilon}{2} \quad \text{and} \quad h(t) \leq \left( M |\beta| \log |\beta| \right)\left( \frac{\epsilon}{2M} \right) |\beta| \sin \left( \frac{1}{|\beta|^2 \log |\beta|} \right) \leq \frac{\epsilon}{2}. \]
Therefore, for all times  
\[ t \in \left[ \log |\beta| + \log\left(\frac{2}{\epsilon} \right), \log |\beta'| - \log\left(\frac{2M}{\epsilon} \right)  \right] \]
the curve $\beta$ has length at most $\epsilon$ on $g_t r_\theta (X, \omega)$
\end{proof}


\begin{cor}
Any angle $\theta \in \cD$ is a divergent on average direction in the moduli space of Riemann surfaces $\M_{g,n}$ (not just in the stratum of quadratic differentials). 
\end{cor}
\begin{proof}
Let $\epsilon > 0$. By the Mumford compactness theorem, $\M_{g,n}$ has a compact exhaustion by sets $K_\epsilon$ of Riemann surfaces on which all simple closed curves have hyperbolic length at least $\epsilon$. By Maskit~\cite{Maskit}, for sufficiently small $\epsilon$, there is an $\epsilon' > 0$ so that $K_\epsilon$ is contained in the set of Riemann surfaces on which all simple closed curves have extremal length at least $\epsilon'$. 

Let $\epsilon'' := \sqrt{c \epsilon'}$ and let $t_n := \log |\beta_n|$. Since $(|\beta_n|)_n$ is an increasing sequence that tends to $\infty$ let $N$ be an integer such that $\frac{4}{(\epsilon'')^2} < \log |\beta_n|$ for $n > N$. By Lemma~\ref{L2}, for all $n > N$ and for all but at most $\log \frac{4M}{(\epsilon'')^2}$ times in $[t_n, t_{n+1}]$ the translation surfaces $\{g_t r_\theta (X, \omega) \}_{t=t_n}^{t_{n+1}}$ contain a cylinder with core curve $\beta_n$ of length less than $\epsilon''$ and of area at least $c$. For these times $\beta_n$ has extremal length at most $\frac{(\epsilon'')^2}{c} = \epsilon'$ and hence the underlying Riemann surface lies outside of $K_\epsilon$.  Since $t_{n+1} - t_n$ tends to $\infty$ as $n \ra \infty$ whereas the amount of time spent in $K_\epsilon$ for times in $[t_n, t_{n+1}]$ is at most $\log \frac{4M}{(\epsilon'')^2}$, we see that $\{g_t r_\theta (X, \omega) \}$ spends asymptotically zero percent of its time in $K_\epsilon$ as desired.
\end{proof}

\begin{proof}[Proof of Theorem~\ref{T1}:]

\noindent Setting $s = \frac{1}{2}$ we see that

\[ \sum_{\beta' \in D_\beta} \frac{\rho_{\beta'}^s |I_{\beta'}|^s}{\rho_\beta^s |I_\beta|^s} = \sum_{\beta' \in D_\beta} \left( \frac{C N_\beta}{C N_{\beta'}} \right)^{1/2} \left( \frac{N_\beta|\beta|^2}{N_{\beta'} |\beta'|^2 } \right)^{1/2} =  \sum_{\beta' \in D_\beta} \frac{N_\beta |\beta|}{N_{\beta'} |\beta'|}. \]

\noindent By (4) of Proposition~\ref{L1}  the sum on the right is greater than
\[ \frac{1}{6L} \sum_{k=0}^{m-1} \sum_{t' \in J_k}  \frac{1}{t'} \]
For each $k$  the smallest value of the inner sum occurs when the $(1-2\nu)2^k\frac{N_\beta}{\delta}$ values of $t'$ in $[2^k\frac{N_\beta}{\delta}, 2^{k+1}\frac{N_\beta}{\delta}]$ are  all exactly distance $1$ apart and  lie in the interval $[\left(1 + 2\nu \right) 2^{k}\frac{N_\beta}{\delta}, 2^{k+1}\frac{N_\beta}{\delta}]$. But then 
\[ \sum_{t \in J_k}  \frac{1}{t}\geq \log (2^{k+1}\frac{N_\beta}{\delta})-\log \left( \left(1 + 2\nu \right) 2^{k}\frac{N_\beta}{\delta} \right) = \log \left( \frac{2}{1+2\nu} \right) \]
Since there are $m$ such sums, we see that
\[ \sum_{\beta' \in D_\beta} \frac{\rho_{\beta'}^s |I_{\beta'}|^s}{\rho_\beta^s |I_\beta|^s} > \frac{m \log \left( \frac{2}{1+2\nu} \right)}{6L} > 1 \]
By Cheung \cite[Theorem 3.3]{Cheung2}, the Hausdorff dimension of $\cD$ is at least $\frac{1}{2}$. Therefore, the Hausdorff dimension of the set of directions that diverge on average is exactly equal to $\frac{1}{2}$ by \cite{AAEKMU}.
\end{proof}

\bibliography{mybib}{}
\bibliographystyle{amsalpha}
\end{document}